\documentclass[a4paper,10pt]{article}

\usepackage[T1]{fontenc}
\usepackage{enumerate}
\usepackage{amsmath,amssymb,amsthm}
\usepackage{mathtools}
\usepackage{graphicx}
\usepackage{bbm}
\usepackage{dsfont}
\usepackage{cite}      

\usepackage[hidelinks,bookmarksnumbered]{hyperref}
\hypersetup{pdfstartview=FitH}
\usepackage{tikz}
\usetikzlibrary{arrows.meta}

  \newcounter{constant}

\def\arraypar#1{\parbox[c]{\textwidth - 2cm}{\centering #1}}

\usepackage{environ}
\NewEnviron{display}{
  \begin{equation}\begin{array}{c}
    \arraypar{\BODY}
  \end{array}\end{equation}
}

\newcommand{\comp}{\mathsf{c}}

\newcommand{\pcb}{\smash{p_c^{\mathrm{b}}}}
\newcommand{\pcs}{\smash{p_c^{\mathrm{s}}}}

\newcommand{\vO}{\smash{\vec{\mathbb{O}}}}
\newcommand{\vZ}{\smash{\vec{\mathbb{Z}}}}

\newcommand{\cC}{\ensuremath{\mathcal{C}}}

\newcommand{\cF}{\ensuremath{\mathcal{F}}}

\newcommand{\cP}{\ensuremath{\mathcal{P}}}

\newcommand{\EE}{\ensuremath{\mathbb{E}}}
\newcommand{\HH}{\ensuremath{\mathbb{H}}}

\newcommand{\NN}{\ensuremath{\mathbb{N}}}

\newcommand{\PP}{\ensuremath{\mathbb{P}}}

\newcommand{\RR}{\ensuremath{\mathbb{R}}}

\newcommand{\ZZ}{\ensuremath{\mathbb{Z}}}

\theoremstyle{plain}
\newtheorem{teo}{Theorem}

\newtheorem{lema}[teo]{Lemma}
\newtheorem{coro}[teo]{Corollary}

\theoremstyle{definition}

\theoremstyle{remark}
\newtheorem{remark}{Remark}

\title{A note on oriented percolation with inhomogeneities and strict inequalities}
\author{
Bernardo N.B. de Lima\footnote{Departamento de Matem\'atica. Universidade Federal
de Minas Gerais, MG, Brazil.
E-mail: bnblima@mat.ufmg.br}
\and
Daniel Ungaretti\footnote{Instituto de Matem\'atica. Universidade Federal
do Rio de Janeiro, RJ, Brazil.
E-mail: daniel@im.ufrj.br}
\and
Maria Eul\'alia Vares\footnote{Instituto de Matem\'atica. Universidade Federal
do Rio de Janeiro, RJ, Brazil. \!Email: eulalia@im.ufrj.br}}

\begin{document}
\maketitle

\begin{abstract}
This note was motivated by natural questions related to oriented percolation on a layered environment that introduces long range dependence. As a convenient tool, we are led to deal with questions on the strict decrease of the percolation parameter in the oriented setup when an extra dimension is added. 

\end{abstract}

\section{Introduction}
\label{sec:introduction}
This note was motivated by a question related to oriented percolation on a naturally occurring graph with layered random inhomogeneities that introduce long range dependence. Among several works that deal with such type of situation, we were particularly inspired by the findings of ~\cite{bramson1991contact, kesten2022oriented,
delima2022dependent, duminil2018brochette} that deal with this kind of problem
for percolation models on planar graphs. Our original motivation had to do with the hexagonal space lattice which we may think of as an infinite version of orange piles that are commonly found in greengrocers. For instance, the fruit seller may be interested in controlling the propagation of fungi through the use of simple artifacts as the introduction of a protection between some of the layers within the pile.

In the course of answering some of the questions that were initially posed, we ended up having to consider how the critical parameter of oriented percolation would be affected by the addition of an extra dimension, e.g. replacing a given graph $\mathbb{G}$ by its corresponding ladder graph $\mathbb{G} \times \mathbb{Z}_+$. As it is well known, general enhancement techniques used in percolation do not extend in a straight manner to the oriented setup, and we had to examine this matter in a more specialized manner. This resulted in findings that we consider to have interest on their own, and that are reported in Section \ref{sec:change-pc}.

The paper is organized as follows: Section \ref{sec:laranjas} explains the model on the hexagonal space lattice and states our results. Section \ref{sec:change-pc} states the main results on the strict decrease of the percolation critical parameter. Section \ref{sec:laranjas_proofs} provides the proofs for the results on the hexagonal space lattice. Some open problems are stated in the concluding remarks.

\section{Inhomogeneous percolation on the hexagonal space lattice}
\label{sec:laranjas}


\medskip
\noindent
\textbf{Hexagonal space lattice.} The graphs that we consider have the hexagonal space lattice $\HH$, as a natural choice of vertex, known to be the densest possible sphere packing in $3\mathrm{d}$-space.

We will focus on oriented (or semi-oriented) percolation and thus we need to choose an
orientation on $\HH$. Depending on the chosen direction we can have square
layers or triangular layers. We found it simpler to analyze when considering
square layers. Having defined $\HH$ as the lattice of $\RR^3$ generated by the vectors
\begin{equation*}
    \vec{u}_1 = (1,0,0)
    \quad \text{and} \quad
    \vec{u}_2 = (0,1,0)
    \quad \text{and} \quad
    \vec{u}_3 = (\tfrac{1}{2},\tfrac{1}{2},\tfrac{1}{\sqrt{2}}),
\end{equation*}
i.e., $\HH = \{a_1 \vec{u}_1 + a_2 \vec{u}_2 + a_3 \vec{u}_3;\; a_1, a_2 \in \ZZ , a_3 \in \ZZ_+\}$, we 
partition it into \textit{layers}: for each $n \in \ZZ$ we define $\HH_n$,
the layer of height $n$ of $\HH$, as $\HH_n := \{a_1 \vec{u}_1 + a_2 \vec{u}_2 + n \vec{u}_3;\;
a_1, a_2 \in \ZZ\}$ and notice that $\HH_n$ is isomorphic to $\ZZ^2$.

We turn $\HH$ into a graph (we will use the notation $\HH$ for the graph as well as for its set of vertices) by connecting all sites at Euclidean distance $1$, which is saying
that if we center a sphere of diameter $1$ at each site of $\HH$ then $v,w \in
\HH$ form an edge if their spheres touch. Hence, any site $v$ has 12 neighbors,
that can be classified according to their height with respect to $v$. If $v \in
\HH_n$, then the set of its neighbors $\Gamma(v)$ is the union of
\begin{align*}
\Gamma^+(v)
    &= \Gamma(v) \cap \HH_{n+1}
    = \{v + \vec{u}_3, v + \vec{u}_3 - \vec{u}_1, v + \vec{u}_3 - \vec{u}_2, v + \vec{u}_3 - \vec{u}_1 - \vec{u}_2 \}; \\
\Gamma^0(v)
    &= \Gamma(v) \cap \HH_{n\phantom{+1}}
    = \{v + \vec{u}_1, v + \vec{u}_2, v - \vec{u}_1, v - \vec{u}_2\}; \\
\Gamma^-(v)
    &= \Gamma(v) \cap \HH_{n-1}
    = \{v - \vec{u}_3, v - \vec{u}_3 + \vec{u}_1, v - \vec{u}_3 + \vec{u}_2, v - \vec{u}_3 + \vec{u}_1 + \vec{u}_2 \},
\end{align*}
the sets of neighbors \textit{above}, \textit{at same height} and
\textit{below} $v$, respectively.

\medskip
\noindent
\textbf{Oriented lattice.} We fix an upward orientation for edges:
for every site $v \in \HH$ we orient edges from $v$ to $\Gamma^{+}(v)$.
Regarding horizontal edges, i.e., edges connecting $v,w \in \HH_n$ for some $n$,
we choose to keep them unoriented. From the orientation choices above,
we define two graphs with vertices on $\HH$:
\begin{itemize}
\item Graph $\vO$, containing only the oriented edges. This makes $\vO$ a
    transitive oriented graph with outdegree $4$ and indegree $4$. Moreover,
    there are straightforward comparisons of $\vO$ with standard oriented graphs.
    Let $\vZ^{d}$ denote the graph with vertices in $\ZZ_+^{d}$ and
    nearest neighbor edges oriented away from the origin, then we can embed $\vZ^{3}$ in $\vO$. For instance we can take
    the sublattice generated by
\begin{equation*}
    w_1 = \vec{u}_3
    \quad \text{and} \quad
    w_2 = \vec{u}_3 - \vec{u}_1
    \quad \text{and} \quad
    w_3 = \vec{u}_3 - \vec{u}_1- \vec{u}_2.
\end{equation*}

\item Graph $\vO_h$, in which we consider both oriented and unoriented (horizontal) edges.
\end{itemize}

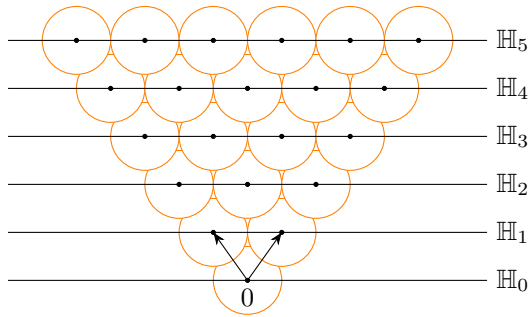
\begin{figure}[ht]
    \centering
    \begin{tikzpicture}[scale=.9]
\def\nrows{5}
\def\latticeheight{0.70711} 
  
  \foreach \row in {0,...,\nrows} {
    \foreach \col in {0,...,\row} {
      \draw[orange,fill=white] (\col-\row/2, \row*\latticeheight) circle (0.5);
      \filldraw (\col-\row/2, \row*\latticeheight) circle (0.03);
    }
      \draw (-3.5,\row*\latticeheight) -- (3.5,\row*\latticeheight) node[right]
      {$\mathbb{H}_{\row}$};
  }

    \draw[-Stealth] (0,0) -- ( 1/2, \latticeheight);
    \draw[-Stealth] (0,0) -- (-1/2, \latticeheight);

    \node at (0,-.25) {$0$};

    \end{tikzpicture}
    \caption{Lateral view of $\HH$, seen as the union of layers $\HH_n$ that are isomorphic to $\ZZ^2$. 
    In oriented graph $\vO$ each site has 4 upward neighbors. In graph $\vO_h$ we allow horizontal (unoriented) edges on each layer. Only sites that can be reached from the origin without horizontal edges are represented.}
    \label{fig:lateral_view}
\end{figure}

\medskip
\noindent
\textbf{Percolation models.}
We consider bond Bernoulli percolation on $\vO_h$ with a random environment.
The model has four parameters: $\delta, p_b, p_g, p_h \in [0,1]$, which we
describe below.

\begin{itemize}
\item The parameter $\delta$ represents the density of defects in our model.
Each layer $\HH_n$ is independently considered either \textit{bad}, with
probability $\delta$, or \textit{good}, with probability $1-\delta$. Upward
edges $vw$ with $v \in \HH_n$ and $w \in \Gamma^+(v)$ are open independently,
conditional on the type of layer:
\begin{itemize}
\item If $\HH_n$ is bad, upward edges are open with probability $p_b \in [0,1]$.
\item If $\HH_n$ is good, upward edges are open with probability $p_g \in [0,1]$, with $p_g \ge p_b$.
\end{itemize}
\item The parameter $p_h$ represents the probability of a horizontal edge to be
    open. In particular, if $p_h = 0$ the model is supported on subgraphs of
    $\vO$.
\end{itemize}

On any graph, we define an \emph{infinite path} starting from the vertex $v_0$ as a sequence of distinct vertices $\langle v_0, v_1, \dots, v_n, \dots\rangle$ such that $\langle v_{n-1}, v_n\rangle$ is an edge of the graph for all $n\in\NN$. As usual in percolation, we use the notation $\{v_0 \to \infty\}$ for the event where there is an infinite open path from the vertex $v_0$, that is, an infinite path such that the edge $\langle v_{n-1}, v_n\rangle$ is open for all $n\in\NN$. Given a fixed origin $o \in G$, the critical parameter for Bernoulli bond percolation is defined as $\pcb(G) := \inf\{p;\; \PP_p(o \to \infty) > 0\}$. Hence, as a consequence of our Corollary~\ref{monotonicidade} we have $\pcb(\vO) \le \pcb(\vZ^3) < \pcb(\vZ^2)$. Analogous definitions hold for Bernoulli site percolation, where the status \textit{open} or \textit{closed} is attached to the sites, and we denote the corresponding critical parameter by $\pcs(G)$.

\medskip
\noindent
\textbf{A KSV-like result:} In~\cite{kesten2022oriented}, Kesten, Sidoravicius
and Vares study essentially the same model for oriented site percolation on
$\vZ^2$ so that there is no $p_h$.
\begin{itemize}
\item[(i)] Their main result is \cite[Theorem 1.1]{kesten2022oriented}, which shows that if $p_g > \pcs(\vZ^2)$ and $p_b > 0$ are fixed and the frequency
    of bad layers is sufficiently small, then the origin percolates with positive
    probability:
    \begin{equation*}
    \text{there is $\delta_0 = \delta_0(p_g, p_b) > 0$ such that}
    \ \PP_{\delta, p_g, p_b}(0 \rightarrow \infty) > 0
    \ \text{for $\delta \in (0, \delta_0)$.}
    \end{equation*}
\item[(ii)] In their Introduction, they also mention that a
    result analogous to that of~\cite{bramson1991contact} does not hold in their setup: there are $p_b > 0$ and $0 < \delta < 1$ such that
    \begin{equation*}
    \text{for any $p_g \in [0,1]$ we have}
    \ \PP_{\delta, p_g, p_b}(0 \rightarrow \infty) = 0.
    \end{equation*}
\end{itemize}
A considerable part of the argument for~\cite[Theorem 1.1]{kesten2022oriented}
relies on planar crossings and a natural follow up question is if a similar
statement can hold in either $\vZ^3$ or $\vO$.
Notice that for $p_g > \pcs(\vZ^2)$ the answer is immediate
from~\cite{kesten2022oriented} (as $\vZ^2$ can be embedded into $\vZ^3$),
but it is not clear if their result can be extended to $p_g \in \bigl(\pcs(\vO), \pcs(\vZ^2)\bigr]$.
This is an interesting open problem.

Regarding (ii), we can prove the same kind of result for $\vO$ (see
Theorem~\ref{teo:subcritical_ph} below), but actually we say a little bit
more.
Recall that our model considers bond percolation on $\vO_h$ with parameters $\delta, p_b, p_g$ and an additional parameter $p_h$. Using a standard
coupling, increasing the value of $p_h$ has the effect of making the open
subgraph more connected, decreasing the damage caused by bad layers.
Hence, it should be expected to see a phase
transition in $p_h$. We present each behavior in a separate result:
\begin{teo}[Subcritical layers]
\label{teo:subcritical_ph}
    Even if $p_g = 1$, given $\delta >0$ and $p_h < \frac{1}{2}$, if $p_b \in
    (0, \frac{\delta^2}{16 \chi(p_h)})$ we have that
\begin{equation}
\label{eq:subcritical_ph}
    \PP_{\delta, p_g, p_b, p_h}(0 \rightarrow \infty) = 0,
\end{equation}
where $\chi(p_h)$ is the expected number of vertices connected to the origin
for independent bond percolation on $\ZZ^2$ with parameter $p_h$.
\end{teo}

Recall that the critical point of Bernoulli bond percolation on $\ZZ^2$ is known to be $1/2$ by the Harris-Kesten Theorem, see~\cite[Chapter 11]{grimmett1998critical}. Therefore, when $p_h > 1/2$ it is obvious that there is percolation with positive probability since one can percolate using only the horizontal edges. 
Regarding the effect of horizontal edges at the critical value $p_h=1/2$, we can see that even if all layers are bad, there is percolation for every positive value of $p_b$.

\begin{teo}[Critical layers, homogeneous]
\label{teo:critical_ph}
Even if $\delta=1$, for $p_h=\tfrac{1}{2}$ and every $p_b > 0$ we have
\begin{equation}
\label{eq:critical_ph}
    \PP_{\delta, p_g, p_b, p_h}(0 \rightarrow \infty) > 0.
\end{equation}
\end{teo}

\section{Strict inequalities for critical points}
\label{sec:change-pc}

This section can be read independently of the remainder of the paper. Its main result is Theorem~\ref{teo:strict_ladder_bond} below, which is used in our proof of Theorem~\ref{teo:critical_ph}.
Let $G$ be any connected graph with bounded degree. We denote by $V(G)$ and $E(G)$ its sets of vertices and edges, respectively. By fixing an orientation for each edge, we make $G$ an oriented graph.
Starting from $G$ we consider its \textit{ladder graph} $G \times \ZZ_+$. In
words, $G \times \ZZ_+$ can be seen as having a collection of copies of $G$, one
for each integer in $\ZZ_+$, connected by edges at corresponding sites of
adjacent copies. More precisely, we have $V(G) \times \ZZ_+$ as the set of
vertices and the set of edges is composed by edges that we divide into two
types:
\begin{itemize}
\item \textit{horizontal edges}, i.e.,  edges connecting $(v,n)$ to $(u, n)$
    where $vu \in E(G)$ and $n \in \ZZ_+$.
\item \textit{vertical edges}, i.e., edges connecting $(v,n)$ to $(v,n+1)$
    for some $n \in \ZZ_+$. Notice that vertical edges are always directed
    upwards.
\end{itemize}

It is also possible to consider graphs $G$ that are not oriented. Following the
same construction will produce a ladder graph with the same set of vertices,
but with horizontal edges that are not oriented. We choose to keep vertical
edges oriented, and this orientation is used in the construction that we
describe next. Some variations of the construction are discussed in
Remarks~\ref{rem:coupling_variations1} and~\ref{rem:coupling_variations2}.

Given $G$ and a parameter $p \in [0,1]$, as mentioned in Section~\ref{sec:laranjas}, there are two models that are usually
considered: Bernoulli site percolation, in which the sites of $G$ can be either
open or closed, with probability $p$ and $1-p$, respectively; or Bernoulli bond
percolation, in which the edges of $G$ can be open or closed, with the same
probabilities. 

Some of the results below are known to be true in a context of non-oriented
edges, using the technique of \emph{enhancements},
see~\cite{aizenman1991strict}. Another interesting work related to strict inequalities for the critical parameter is \cite{martineau2019covering}. One interesting aspect of the results below is
that they are applicable to models of oriented percolation for which standard
enhancements do not work. In particular, as a consequence of
Theorems~\ref{teo:strict_ladder_bond} and~\ref{teo:strict_ladder_site} below, one has the 
following strict monotonicity that, as far as we know, is not found in the literature:
\begin{coro}
\label{monotonicidade}
For any $d \ge 2$ the critical point for Bernoulli oriented bond
percolation in $\vZ^d$ is strictly decreasing as a function of $d$. The same is
true for oriented site percolation.
\end{coro}

\medskip
\noindent
\textbf{Bond percolation.} 
Firstly, we prove a result for bond percolation.
Let $\pcb(G)$ denote the critical point of Bernoulli bond percolation on a graph
$G$. It is immediate that $\pcb(G) \le \pcb(G \times \ZZ_+)$. We show that the
inequality is actually strict.
\begin{teo}
\label{teo:strict_ladder_bond}
Let $G$ be a connected graph with bounded degree and $\pcb(G) < 1$.
It holds that $$\pcb(G \times \ZZ_+) < \pcb(G).$$
\end{teo}

The proof is based on a monotonic coupling of the two percolation processes.
Such couplings are not new in the literature, see
e.g.~\cite{grimmett1998critical}. Our construction is quite close to that in~\cite{gomes2023anisotropic}.  The idea of the proof is rather
simple: we leverage the extra dimension given by the vertical edges of $G
\times \ZZ_+$ in order to strictly improve the probability of percolating.

\begin{proof}[Proof of Theorem~\ref{teo:strict_ladder_bond}]
Distinguish some vertex $o \in G$ by considering it as an origin and let the
origin of $G \times \ZZ_+$ be $(o,0)$. We explore the clusters of $o$ in $G$
and $(o,0)$ in $G \times \ZZ_+$ simultaneously, via a dynamic coupling of
Bernoulli bond percolation on $G$ and $G \times \ZZ_+$ with different
parameters.  More precisely, we consider a slight variation of graph $G
\times \ZZ_+$ that we denote $\tilde{G}$.  Denote the maximum degree of $G$
by $\Delta$. Graph $\tilde{G}$ is obtained from $G\times \ZZ_+$ by splitting
each vertical edge $\tilde{e}$ into edges $\tilde{e}_j, 1 \le j \le \Delta$,
i.e., $\Delta$ parallel edges with the same endpoints.

Let $p \in (0,1)$ be the probability that an edge is open on $G \times \ZZ_+$.
On $\tilde{G}$, each parallel edge $\tilde{e}_j$ with $1 \le j \le \Delta$
is open independently with probability $\tilde{p}$ satisfying $1-p =
(1-\tilde{p})^{\Delta}$. This choice can be interpreted as saying that a
vertical edge $\tilde{e}$ is closed if and only if every edge $\tilde{e}_j$ is
closed and hence it is immediate that percolating from $(o,0)$ in $G\times \ZZ_+$ or
$\tilde{G}$ are equivalent. This step ensures independence in
the exploration process described below.

Consider a fixed order on the edges of $G$. We inductively define some objects
that register the information collected up to step $n$. A glimpse of them is
given below, and then we explain them further.
\begin{description}
\item[$(e_n)_n$:] bonds of $G$ in the order that they were explored.
\item[$(w_n)_n$:] vertices of $G$ in the cluster of $o$, in the order that they
    were explored.
\item[$(\tilde{w}_n)_n$:] vertices of $\tilde{G}$ in the cluster of $(o,0)$,
    corresponding to $w_n$ at some height $h_n$: $\tilde{w}_n = (w_n, h_n)$.
\item[$(A_n)_n$:] the cluster of $o$ observed up to step $n$.
\item[$(B_n)_n$:] the set of bonds of $G$ explored up to time $n$, i.e.,
    $B_n = \{e_j;\; 1 \le j \le n\}$.
\end{description}
Given a set of vertices $A$ in $G$, we denote its
exterior edge boundary as the subset of edges in $G$:
\begin{equation*}
\partial^e A=\{\langle v,w\rangle:v\in A\text{ and }w\notin A\}.
\end{equation*}
In the beginning, we set $w_0=o$ and $\tilde{w}_0=(o,0)$, implying $A_0=\{o\}$
and $B_0=\emptyset$. For $n\geq 1$, given $(w_j)_{j=0}^{n-1}$,
$(\tilde{w}_j)_{j=0}^{n-1}$, $A_{n-1}$ and $B_{n-1}$ we have two cases:
\begin{enumerate}[(1)]
\item If $\partial^e A_{n-1}\cap B_{n-1}^c=\emptyset$, we stop the procedure
    and declare that the origin does not percolate.

\item If $\partial^e A_{n-1}\cap B_{n-1}^c\neq\emptyset$, define
    $e_n:=\min (\partial^e A_{n-1}\cap B_{n-1}^c)$ and $B_n := B_{n-1} \cup
    \{e_n\}$. We can write $e_n=\langle w_{j},v\rangle$, with $j\leq n-1$,
    $w_j\in A_{n-1}$ and $v\notin A_{n-1}$.\label{item:exp_continues}
\end{enumerate}
In case \eqref{item:exp_continues}\textcolor{orange}{,} we proceed with the exploration.
Vertex $w_j$ that composes $e_n$ has a corresponding vertex in $\tilde{G}$ at
some height $h_j$, that is, $\tilde{w}_j = (w_j, h_j)$. Moreover, let
$\tilde{e}_{j,n}$ be one of the parallel edges in $\tilde {G}$ from $(w_j, h_j)$ to $(w_j, h_j+1)$.
Since $G$ has maximum degree $\Delta$, in our construction we can always choose
some parallel edge that has not been used yet.
We say that edge $e_n=\langle w_j, v \rangle$ is \emph{red}  if one of the conditions below holds:
\begin{enumerate}[(i)]
\item The bond $\langle \tilde{w}_j, (v, h_j)\rangle$ is open; then, we set
    $w_n=v$ and $\tilde{w}_n=(v, h_j)$.

\item The bond $\langle \tilde{w}_j, (v, h_j)\rangle$ is closed but both
    $\tilde{e}_{j,n}$ and $\langle (w_j, h_j + 1), (v, h_j + 1)\rangle$ are open;
    then we set $w_n=v$ and $\tilde{w}_n=(v,h_j+1)$.
\end{enumerate}

If $e_n$ is red, we set $A_n=A_{n-1} \cup \{w_n\}$; otherwise, $A_n=A_{n-1}$.
Observe that the event $\{e_n\text{ is red}\}$ is independent of the previous
steps and its probability is $f(p) := p+(1-p)\tilde{p}p$. This means that the
cluster of the origin connected by red paths is distributed as a cluster of
Bernoulli bond percolation with parameter $f(p)$. By construction,
an infinite red cluster in $G$ is associated to an infinite cluster from $(o,0)$
in $\tilde{G}$, and therefore the same holds in $G \times \ZZ_+$. Thus, we have
that
\begin{equation}
\label{eq:coupling_alcinhas}
    \text{if $f(p) > \pcb(G)$ then $p \ge \pcb(G\times \ZZ_+)$.}
\end{equation}
It is clear that $f(\pcb(G)) > \pcb(G)$ and since $f$ is a continuous function
of $p$, we can find $p' < \pcb(G)$ such that $f(p') > \pcb(G)$, implying by
\eqref{eq:coupling_alcinhas} that $\pcb(G) > p' \ge \pcb(G\times \ZZ_+)$.
\end{proof}

\medskip
\noindent
\textbf{Site percolation.} A natural follow up question is if the same
kind of coupling used in Theorem~\ref{teo:strict_ladder_bond} can be done in the context of
Bernoulli site percolation. The theorem below gives the answer.

\begin{teo}
\label{teo:strict_ladder_site}
Let $G$ be a connected graph with bounded degree and $\pcs(G) < 1$.
It holds that $$\pcs(G \times \ZZ_+) < \pcs(G).$$
\end{teo}

\begin{proof}
The proof is based on an exploration that is similar to the one in the proof of
Theorem~\ref{teo:strict_ladder_bond}, so we point out only the main
differences, leaving out some of the details. The main difference is that
in site percolation we do not have a clear distinction between vertical
edges and horizontal edges.

Instead of working with $G \times \ZZ_+$, we consider a graph $\tilde{G}$ in
which every site $v$ is replaced by $\Delta$ `parallel' sites
$\{\tilde{v}_j;\; 1 \le j \le \Delta\}$, with each site $\tilde{v}_j$ being open independently
with probability $\tilde{p}$ satisfying $1-p = (1-\tilde{p})^{\Delta}$.
Starting from $w_0 = o$, $\tilde{w}_0 = (o,0)$, $A_0=\{o\}$ and $B_0=\emptyset$, we inductively build the
cluster of $o$ at step $n$, denoted $A_n$: we choose some site
$v \in \partial^{e}_* A_{n-1}\cap B_{n-1}^c$ that still has not been explored where $\partial^{e}_* A=\{v\in A^c : \text{ there exists an edge }\langle u,v\rangle\text{ with }u\in A\}$ is the exterior site boundary of $A$. By definition,
$v \notin A_{n-1}$ and has a neighbor $w_j \in A_{n-1}$; moreover, $w_j$ has a
corresponding site $\tilde{w}_j = (w_j,h_j)$ in $\tilde{G}$. We say that $v$ is
\textit{red} if one of the conditions below holds:
\begin{enumerate}[(i)]
\item 
    Site $(v,h_j)$ is open (meaning all parallel sites $(v,h_j)_k$ are open);
    then, we set $w_n=v$ and $\tilde{w}_n=(v, h_j)$.

\item Site $(v,h_j)$ is closed but $(w_j,h_j+1)_k$ (with $k$ referring to the smallest unexplored
    parallel site) is open and site $(v,h_j+1)$ is open (again, all
    parallel copies); then we set $w_n=v$ and $\tilde{w}_n=(v,h_j+1)$.
\end{enumerate}
If $v$ is not red, we let $B_n=B_{n-1}\cup\{v\}$ be the set of explored sites and never test it
for being red again. As in Theorem~\ref{teo:strict_ladder_bond}, the probability
of a new explored site being red is $f(p)$. The independence in the construction
is guaranteed by (ii), since $w_j$ has at most $\Delta$ neighbors. From here
on, the proof is the same as in Theorem~\ref{teo:strict_ladder_bond}.
\end{proof}

\begin{remark}
\label{rem:coupling_variations1}
Some results for bond percolation can be obtained by studying site
    percolation on the covering graph, see~\cite[p. 24]{Grimmett}. However, we
    emphasize that the operations of taking the covering graph and taking 
    the ladder graph do not commute in general, so we actually study site and bond percolation separately.
\end{remark}

\begin{remark}
\label{rem:coupling_variations2}
Instead of ladder graphs, one could consider replacing $\ZZ_+$ by
    $\ZZ_k$, the integers modulo $k$. It is straightforward to check that
    Theorem~\ref{teo:strict_ladder_bond} is true for $G \times \ZZ_k$ for $k\geq 2$ and
    Theorem~\ref{teo:strict_ladder_site} is true for $G \times \ZZ_k$ for $k\geq 3$.
\end{remark}

\begin{remark}
\label{rem:coupling_variations3}
Given the situation in Corollary \ref{monotonicidade}, it is very natural to ask whether we can say the same for $\lambda_c(d)$ the critical parameter of Harris contact process \cite{liggett1985interacting}. We cannot at this point give an answer. What we can say is that a minor modification of our argument applies to the following  discrete time approximation: Take $\ZZ^d \times \ZZ_+$ as the vertex set of the graph and consider the following oriented percolation model: if $m, m'$ are nearest neighbors in $\ZZ^{d}$ the (upwards) oriented edges $\langle (m,n), (m', n+1) \rangle$ are open with probability $p$, and the vertical (upwards) oriented edges $\langle (m,n), (m, n+1) \rangle$ are open with probability $1-\delta$. In particular, if $d=2$ and $\delta=1$ this is the same as the oriented percolation on $\vO$.
For each fixed $\delta$, there is a phase transition in $p$ and the critical point $p_c(d,\delta)$ is strictly decreasing in $d$, for each $\delta$. Another question that might be treatable by similar arguments: Is $p_c(\vO) < p_c(\vZ^3)$?
\end{remark}


\section{Proofs for hexagonal space lattice model}
\label{sec:laranjas_proofs}

\subsection{Subcritical layers}
\label{sub:subcritical_layers}

Before proving Theorem~\ref{teo:subcritical_ph}, we will need an auxiliary
lemma. Denote by $\Gamma^{\uparrow}(0)$ the collection of all sites that can
be reached from the origin in $\vO$, i.e., using only upward edges. It is clear that
$\Gamma^{\uparrow}(0)$ is a pyramid whose intersection with layer $\HH_n$
is a square with $(n+1)^2$ sites:
\begin{equation*}
\Gamma^{\uparrow}(v) \cap \HH_n
    = v + \bigl\{a_1 \vec{u}_1 + a_2 \vec{u}_2 + n \vec{u}_3;
    - n \le a_i \le 0,\ \text{for $i=1,2$}\bigr\},
    \quad \text{for $n \ge 0$}.
\end{equation*}
Suppose that $p_g=1$ and there are no bad layers and no horizontal edges.
Then, the set of sites connected to the origin is precisely
$\Gamma^{\uparrow}(0)$. If we add the possibility of using open horizontal
edges, then the set $C$ of sites connected to the origin will be larger and random.
However, since $p_h < \frac{1}{2}$ is subcritical for bond percolation on $\ZZ^2$,
the set $C \cap \HH_n$ cannot grow too fast.

\begin{lema}
\label{lema:cone_growth}
Fix $p_g = 1$ and $p_h < \frac{1}{2}$. If $C$ is the cluster of the origin for
percolation in $\vO$ then $C \cap \HH_n$ is eventually contained in a ball
with deterministic radius $n(\ln n)^2$, almost surely.
\end{lema}

\begin{proof}
The proof follows closely the arguments
in~\cite[Section~4.2]{hilario2022results}, which controls the growth of a
similar process. We provide some of the details for completeness.

Since $p_g=1$, all upward edges are open.
For each layer $\HH_n$, denote the collection of open horizontal edges by
$\cP_n$. For $V \subset \HH_n$ let $\cC_n(V)$ be the connected
component of $V$ in $\cP_n$ by open edges. Since $p_h < \frac{1}{2}$,
almost surely $\cC_n(V)$ is finite for finite $V$. Define $C_0 := \cC_0(0)$,
the sites connected to the origin by open horizontal edges.
We define sets $C_n$ for $n \ge 1$ by induction:
\begin{equation}
\label{eq:defi_Cn}
C_n = \cC_n(\Gamma^{+}(C_{n-1})).
\end{equation}
In words, from set $C_{n-1}$ we consider all their upward neighbors and define
$C_n$ as the set of all sites that can be reached in $\HH_n$ from such
neighbors via open horizontal edges. We have that $C \cap \HH_n = C_n$.
Define
\begin{equation*}
    R_n := \max\{|x|+|y|;\; (x,y,z) \in C_n\}
    \quad \text{and} \quad
    B_n(r) := \{(x,y,z) \in \HH_n;\; |x| + |y| \le r\}.
\end{equation*}
Since $C \supset \Gamma^{\uparrow}(0)$, it is clear that $R_n \ge n$.
Moreover, we can relate $R_{n+1}$ and $R_n$ in the following way.
Notice that $C_{n} \subset B_{n}(R_{n})$ implies that
$\Gamma^{+}(C_{n}) \subset B_{n+1}(R_n + 1)$. Hence, we can write that
$C_{n+1} \subset \cC_{n+1}(B_{n+1}(R_n + 1))$. Define the events
\begin{equation*}
A_{n+1} := \{R_{n+1} \ge R_n + \eta \ln R_n\}.
\end{equation*}
for some constant $\eta(p_h) > 0$ that is chosen below.
Consider the filtration $\cF_n := \sigma(\cP_{i};\; i \le n)$ and notice that
$A_n \in \cF_n$. Given $\cF_{n}$ we have that on the event $A_{n+1}$ there must
be some point $x \in \partial B_{n+1}(R_n+1)$ that satisfies
$x \leftrightarrow x + \partial B(\eta \ln R_n)$
in percolation $\cP_{n+1}$. Hence, exponential decay of cluster size gives
the estimate
\begin{equation*}
\PP(A_{n+1} \mid \cF_{n})
    \le \sum_{x \in \partial B_{n+1}(R_n+1)}
        \PP(x \leftrightarrow x + \partial B_{n+1}(\eta \ln R_n) \mid \cF_{n})
    \le c R_{n} e^{- \psi(p_h) \eta \ln R_n},
\end{equation*}
for some positive constant $\psi(p_h) >0$, see
e.g.~\cite[Theorem~(6.75)]{Grimmett}.
Choose $\eta(p_h) := \frac{3}{\psi(p_h)}$, which leads to
$\PP(A_{n+1} \mid \cF_{n}) \le c R_{n}^{-2}$.
The linear growth estimate $R_n \ge n$ and the fact that
$x \mapsto c x^{-2}$ is decreasing imply that
\begin{equation*}
\sum_{n \ge 1} \PP(A_{n+1} \mid \cF_{n})
    \le c \sum_{n \ge 1} n^{-2}
    < \infty.
\end{equation*}
Using a conditional Borel-Cantelli lemma,
see~\cite[Theorem~5.3.2]{durrett2019probability}, we have that $\PP(\varlimsup A_n) = 0$ implying
that there is a random $n_1$ such that $R_{n+1} \le R_n + \eta \ln R_n$ for
$n \ge n_1$. This implies estimates on the growth of $R_n$, by the same
reasoning as in~\cite{hilario2022results}, around Equation~(4.6):

\begin{equation}
\label{eq:growth_Rn}
\text{for any $a>1$ fixed we have}\
    \varlimsup_n \frac{R_n}{n(\ln n)^{a}} \le 1,\ \text{a.s.}
\end{equation}
For simplicity we take $a=2$.
\end{proof}

\begin{proof}[Proof of Theorem~\ref{teo:subcritical_ph}]
The proof is based on a Borel-Cantelli argument.
Lemma~\ref{lema:cone_growth} controls the growth of $\cC \cap \HH_n$.
Now, we will introduce the structure of bad layers and show that
with a reasonable frequency we will bump into a sequence of many consecutive
bad layers. Choosing $p_b$ sufficiently small will imply then that almost
surely the origin will not percolate. The idea is to look for the
occurrence of $n$ consecutive bad layers in some region.

Let us define the sequence of heights $(c_n; \; n\ge 1)$ by
\begin{equation*}
    c_n := \Bigl(\frac{2}{\delta}\Bigr)^n
\end{equation*}
Define
\begin{equation*}
A_n := \{\text{there are $n$ consecutive bad layers $\HH_{j}, \cdots,
    \HH_{j+n-1}$ with $j \in [c_n, c_{n+1} - n)$}\}.
\end{equation*}
Dividing $[c_n, c_{n+1} - n)$ into blocks of length $n$, we have by
independence that
\begin{equation*}
\PP(A_n^{\comp})
    \le (1-\delta^n)^{\frac{c_{n+1}-c_n}{n}}
    \le \exp\Bigl[ - \frac{\delta^n}{n}  \Bigl[\Bigl(\frac{2}{\delta}\Bigr)^{n+1} - \Bigl(\frac{2}{\delta}\Bigr)^n\Bigr]\Bigr]
    = \exp\Bigl[ - \frac{2^n}{n}  \Bigl[\Bigl(\frac{2}{\delta}\Bigr) - 1\Bigr]\Bigr]
\end{equation*}
which clearly satisfies $\sum \PP(A_n^{\comp}) < \infty$. Therefore, by
Borel-Cantelli's Lemma almost surely there is $n_0$ such that $A_n$ happens for every
$n \ge n_0$. In other words, for each $n \ge n_0$ we can choose a block $L_n$ of $n$
consecutive bad layers whose heights are given in order by $l_{n,j}$ with
$1 \le j \le n$ and $c_n \le l_{n,1} < l_{n,n} < c_{n+1}$.

Recall that by Lemma~\ref{lema:cone_growth} we almost surely
have $n_1$ sufficiently large such that $C_n \subset B_n\bigl(n (\ln n)^2\bigr)$
for every $n \ge n_1$. Thus, for $n \ge \max\{n_0, n_1\}$ we know that all
vertices from the first bad layer of $L_{n}$ that can be reached from the origin 
are contained in a ball, more precisely,
\begin{equation}
\label{eq:vertices_first_layer}
C_{l_{n,1}}
    \subset B_{l_{n,1}}\bigl(l_{n,1} (\ln l_{n,1})^2\bigr)
    \subset B_{l_{n,1}}\bigl(c_{n+1} (\ln c_{n+1})^2\bigr).
\end{equation}
In order to cross $L_n$ we have to find, for each $1 \le j \le n$, oriented
edges $\vec{e}_j = u_j v_j$ that connect height $l_{n,j}$ to
the layer above and satisfying
\begin{equation}
\label{eq:available_cond}
 u_1 \in C_{l_{n,1}}   
\quad \text{and} \quad
\text{$u_j \in
\cC_{l_{n,j}}(v_{j-1})$ for $2 \le j \le n$.}
\end{equation}
We say that a sequence $\vec{e}_1, \ldots, \vec{e}_n$ is \textit{available} if it satisfies~\eqref{eq:available_cond}. To make it clear, we emphasize that for a sequence of $\vec{e}_j$ to be available it is not needed to check if  the edges $\vec{e}_j$ are open or closed. 
Denote by $N_n$ the number of available sequences. Since each $u_j$ has 4 neighbors above, choosing $u_1, v_1, u_2, v_2, \ldots, v_n$ in order, we can estimate $N_n$ by
\begin{equation*}
N_n \le \# C_{l_{n,1}} \cdot 4 \cdot
    \prod_{j=2}^{n} (\# \cC_{l_{n,j}}(v_{j-1}) \cdot 4)
    = 4^n \cdot \# C_{l_{n,1}} \cdot \prod_{j=2}^{n} \# \cC_{l_{n,j}}(v_{j-1}).
\end{equation*}
If there is a crossing of the block $L_n$, then $N_n \ge 1$ and all edges
of paths in $N_n$ must be open. We bound the
probability of $N_n \ge 1$ via a first moment estimate. Denote by $\cF_n$ the
$\sigma$-algebra generated by all edges with vertices in $\HH_j$ with $j \le n$.
First notice that conditioning on $\cF_{l_{n,n-1}}$ we have that
\begin{align*}
\EE N_n
    &\le 4^{n} \EE \Bigl[\# C_{l_{n,1}} \prod_{j=2}^{n-1} \#
    \cC_{l_{n,j}}(v_{j-1}) \cdot
    \EE \Bigl[\# \cC_{l_{n,n}}(v_{n-1}) \Bigm| \cF_{l_{n,n-1}}\Bigr] \Bigr]\\
    &= 4^{n} \chi(p_h) \cdot
    \EE \Bigl[\# C_{l_{n,1}} \prod_{j=2}^{n-1} \# \cC_{l_{n,j}}(v_{j-1})
    \Bigr],
\end{align*}
by independence. Hence, conditioning successively on $\cF_{l_{n,j}}$ with
$j = n-1, \ldots, 2, 1$ we conclude that
\begin{equation*}
\EE N_n
    \le 4^{n} \chi(p_h)^{n-1} \cdot \EE \Bigl[\# C_{l_{n,1}} \Bigr]
    \le 4^{n} \chi(p_h)^{n-1} \cdot \# B_{l_{n,1}}\bigl(c_{n+1} (\ln
    c_{n+1})^2\bigr).
\end{equation*}
Finally, by the discussion above, when $n \ge \max\{n_0, n_1\}$, on the event
that the origin percolates we must cross all blocks $L_n$ for large $n$. The
probability of crossing some $L_n$ is at most
\begin{equation*}
4^{n} \chi(p_h)^{n-1} \cdot \# B_{l_{n,1}}\bigl(c_{n+1} (\ln c_{n+1})^2\bigr)
    \cdot (p_b)^n
\end{equation*}
since for any choice of $\vec{e}_j$ the probability of an open path passing by
all $\vec{e}_j$ is at most $(p_b)^n$. Also,
\begin{equation*}
\# B_{l_{n,1}}\bigl(c_{n+1} (\ln c_{n+1})^2\bigr)
    \sim K \cdot \Bigl(
        \Bigl(\frac{2}{\delta}\Bigr)^{n+1} \!\!
        (n+1)^2 \Bigl(\ln \frac{2}{\delta}\Bigr)^2
    \Bigr)^2
    \sim K \cdot (n+1)^4 \Bigl(\frac{4}{\delta^2}\Bigr)^{n+1}
\end{equation*}
for some positive constant $K = K(\delta)$. Consequently, the probability of
crossing $L_n$ decreases to zero when
\begin{equation*}
    \frac{16 \chi(p_h) p_b}{\delta^2} < 1.
\end{equation*}
The result follows since each attempt of crossing some $L_n$ is an independent try.
\end{proof}


\subsection{Critical layers}
\label{sub:critical_layers}
\noindent
\begin{proof}[Proof of Theorem~\ref{teo:critical_ph}] Initially, we can observe that, if we delete in $\mathbb{H}$ all bonds $$\{\langle v,v + \vec{u}_3 - \vec{u}_1\rangle, \langle v,v + \vec{u}_3 - \vec{u}_2\rangle, \langle v,v + \vec{u}_3 - \vec{u}_1  - \vec{u}_2\rangle : v\in\HH\},$$ 
the resulting graph is isomorphic to $\ZZ^{2}\times \ZZ_+$, the cubic lattice with oriented bonds along the third coordinate axis. Hence, it is enough to prove that there is percolation on $\ZZ^{2}\times \ZZ_+$ where vertical bonds are open with probability $p_b > 0$ and the non-oriented bonds are open with probability $\tfrac{1}{2}$, independently of each other. A minor adaptation of the proof of Theorem~\ref{teo:strict_ladder_bond} yields the conclusion.
\end{proof}


\section*{Acknowledgements}
\noindent
The authors thank Lucas Vares Vargas for asking the question that motivated us for this work and E. Andjel for useful discussions. B.N.B.L. would like to thank the probability group of UFRJ for the hospitality in several visits.  M.E.V. is partially supported by CNPq, Brazil grant 310734/2021-5 and by FAPERJ, Brazil grant E-26/200.442/2023. Visits of B.N.B.L. to UFRJ were partially supported by FAPERJ, Brazil grants E-26/202.636/2019 and E-26/200.442/2023.

\bibliographystyle{plain}
\bibliography{laranjas}
\end{document}